\documentclass[a4paper,oneside,reqno]{amsart}
\usepackage[utf8]{inputenc}
\usepackage[a4paper,margin=3cm]{geometry} 
\usepackage{bm}
\usepackage{amsmath}
\usepackage{amsthm}
\usepackage{amscd}
\usepackage{amssymb}
\usepackage{MnSymbol}
\usepackage{latexsym}
\usepackage{eucal}
\usepackage{dsfont}
\usepackage{mathtools}
\usepackage{enumitem}
\usepackage{todonotes}

\usepackage{verbatim}
\usepackage{graphicx}
\usepackage{float}
\usepackage{array,booktabs}

\usepackage[colorlinks,pdftex]{hyperref}
\hypersetup{
linkcolor=black,
citecolor=black,
pdftitle={},
pdfauthor={Franziska K\"uhn},
pdfkeywords={},
}

\swapnumbers
\theoremstyle{theorem} \newtheorem{thm}{Theorem}[section]
\theoremstyle{theorem} \newtheorem{lem}[thm]{Lemma}
\theoremstyle{theorem} \newtheorem{prop}[thm]{Proposition}
\theoremstyle{theorem} 
\theoremstyle{definition} \newtheorem{defn}[thm]{Definition}
\theoremstyle{remark} 
\theoremstyle{remark} 
\theoremstyle{definition} 
\theoremstyle{remark} \newtheorem{bem}[thm]{Remark}
\theoremstyle{remark} 
\theoremstyle{definition}  
\theoremstyle{definition}  
\theoremstyle{definition} 
\theoremstyle{definition} \newtheorem*{ack}{Acknowledgement}

\DeclareMathOperator \re {Re}

\DeclareMathOperator \spt {supp}

\DeclareMathOperator \tr {tr}

\DeclareMathOperator \loc {loc}

\newcommand\sprod[1]{\langle#1\rangle}

\newcommand{\I}{\mathds{1}}
\newcommand\floor[1]{\left\lfloor #1 \right\rfloor}
\newcommand\fa{\qquad \text{for all \ }}

\newcommand{\cadlag}{c\`adl\`ag }

\newcommand\mc[1] {\mathcal{#1}}
\newcommand\mbb[1] {\mathds{#1}}

\newcommand{\eps}{\varepsilon}

\makeatletter
\DeclareRobustCommand\widecheck[1]{{\mathpalette\@widecheck{#1}}}
\def\@widecheck#1#2{%
    \setbox\z@\hbox{\m@th$#1#2$}%
    \setbox\tw@\hbox{\m@th$#1%
       \widehat{%
          \vrule\@width\z@\@height\ht\z@
          \vrule\@height\z@\@width\wd\z@}$}%
    \dp\tw@-\ht\z@
    \@tempdima\ht\z@ \advance\@tempdima2\ht\tw@ \divide\@tempdima\thr@@
    \setbox\tw@\hbox{%
       \raise\@tempdima\hbox{\scalebox{1}[-1]{\lower\@tempdima\box
\tw@}}}%
    {\ooalign{\box\tw@ \cr \box\z@}}}
\makeatother

\begin{document}

\title{A Liouville theorem for L\'evy generators}
\author[F.~K\"{u}hn]{Franziska K\"{u}hn} 
\address[F.~K\"{u}hn]{TU Dresden, Fachrichtung Mathematik, Institut f\"{u}r Mathematische Stochastik, 01062 Dresden, Germany.}
\email{franziska.kuehn1@tu-dresden.de}

\begin{abstract}
	Under mild assumptions, we establish a Liouville theorem for the ``Laplace'' equation $Au=0$ associated with the infinitesimal generator $A$ of a L\'evy process: If $u$ is a weak solution to $Au=0$ which is at most of (suitable) polynomial growth, then $u$ is a polynomial. As a by-product, we obtain new regularity estimates for semigroups associated with L\'evy processes.
\end{abstract}

\maketitle

\section{Introduction} \label{intro}

The classical Liouville theorem states that any bounded solution $u:\mbb{R}^d \to \mbb{R}$ to the Laplace equation $\Delta u=0$ is constant. There is an extension for unbounded functions: If $\Delta u=0$ and $u$ is at most of polynomial growth, say, $|u(x)| \leq C(1+|x|^k)$ for some constants $C>0$ and $k \in \mbb{N}_0$, then $u$ is a polynomial of degree at most $k$. In this paper, we extend this result to a wide class of integro-differential operators. More precisely, we establish a Liouville theorem for equations $Au=0$ where $A$ is of the form \begin{equation*}
	Af(x) = b \cdot \nabla f(x) + \frac{1}{2} \tr(Q \cdot \nabla^2 f(x)) + \int_{y \neq 0}(f(x+y)-f(x)-y \cdot \nabla f(x) \I_{(0,1)}(|y|)) \, \nu(dy), \,\, f \in C_c^{\infty}(\mbb{R}^d),
\end{equation*}
for some $b \in \mbb{R}^d$, a positive semi-definite matrix $Q \in \mbb{R}^{d \times d}$ and a measure $\nu$ on $(\mbb{R}^d \backslash \{0\},\mc{B}(\mbb{R}^d \backslash \{0\}))$ satisfying $\int_{y \neq 0} \min\{1,|y|^2\} \, \nu(dy)<\infty$. Equivalently, $A$ can be written as a pseudo-differential operator, \begin{equation}
	Af(x) = - \psi(D)f(x):= -\int_{\mbb{R}^d} \psi(\xi) e^{ix \cdot \xi} \hat{f}(\xi) \, d\xi, \qquad f \in C_c^{\infty}(\mbb{R}^d),\, x \in \mbb{R}^d, \label{pseudo}´
\end{equation}
where $\hat{f}(\xi)=(2\pi)^{-d} \int_{\mbb{R}^d} f(x) e^{-ix \cdot \xi} \, dx$ denotes the Fourier transform of $f$ and  the symbol $\psi$ is a continuous negative definite function with L\'evy--Khintchine representation\begin{equation}
	\psi(\xi) = i b \cdot \xi + \frac{1}{2} \xi \cdot Q\xi + \int_{y \neq 0} \left(1-e^{iy \cdot \xi} + iy \cdot \xi \I_{(0,1)}(|y|) \right) \, \nu(dy), \qquad \xi \in \mbb{R}^d. \label{char-exp}
\end{equation}
Since $A$ is the infinitesimal generator of a L\'evy process, see below, we also call $A$ a L\'evy generator. The family of L\'evy generators includes many interesting and important operators, e.g.\ the Laplacian $\Delta$, the fractional Laplacian $-(-\Delta)^{\alpha/2}$, $\alpha \in (0,2)$, and the free relativistic Hamiltonian $m-\sqrt{-\Delta+m^2}$, $m>0$. If $A$ is a local operator, i.e.\ $\nu=0$, then the Liouville theorem is classical, and so the focus is on the non-local case $\nu \neq 0$. For L\'evy generators with a sufficiently smooth symbol, there is a Liouville theorem by Fall \& Weth \cite{fall16}; the required regularity of $\psi$ increases with the dimension $d \in \mbb{N}$. Ros-Oton \& Serra \cite{ros-oton16} established a general Liouville theorem for symmetric stable operators, \begin{equation*}
	Af(x) = \int_{\mbb{S}^{d-1}}\int_{(0,\infty)}  \left( f(x+\theta r)+f(x-\theta r)-2f(x)\right) \frac{dr}{r^{d+\alpha}} \, \mu(d\theta), \qquad f \in C_c^{\infty}(\mbb{R}^d), \, x \in \mbb{R}^d,
\end{equation*}
where $\alpha \in (0,2)$ and $\mu$ is a non-negative finite measure on the unit sphere $\mbb{S}^{d-1}$ satisfying an ellipticity condition. The recent papers \cite{jakobsen,knop19} give necessary and sufficient conditions for the Liouville property, i.e.\ conditions under which the implication \begin{equation}
	u \in L^{\infty}(\mbb{R}^d), \, Au=0 \, \, \text{weakly} \implies \text{$u$ is constant} \label{liou-p}
\end{equation}
holds. Choquet \& Deny \cite{choquet60} characterized the bounded solutions $u$ to convolution equations of the form $u=u \ast \mu$; these equations play a central role in the study of the ``Laplace'' equation $Au=0$, see Lemma~\ref{p-3}. Since the Liouville theorem is an assertion on the smoothness of harmonic functions, there is a close connection between the Liouville theorem and Schauder estimates; see \cite{reg-levy,ros-oton16} and the references therein for recent results. We would like to mention that there are also Liouville theorems in the half-space, see e.g.\ \cite{cheng15,ros-oton16}, and  Liouville theorems for certain L\'evy-type operators, see e.g.\ \cite{barlow00,uemura11,priola04,wang12}. 

In this paper, we use a probabilistic approach, inspired by \cite{ros-oton16}, to prove a Liouville theorem for a wide class of L\'evy generators. Before stating the result, let us briefly recall some material from probability theory. It is well known, cf.\ \cite{sato,jacob1,barcelona}, that there is a one-to-one correspondence between continuous negative definite functions and L\'evy processes, i.e.\ stochastic processes with \cadlag (right-continuous with finite left-hand limits) sample paths and stationary and independent increments. Given a continuous negative definite function $\psi:\mbb{R}^d \to \mbb{C}$, there exists a L\'evy process $(X_t)_{t \geq 0}$ with semigroup $P_t f(x):= \mbb{E}f(x+X_t)$ satisfying
\begin{equation*}
	-\psi(D)f(x)= \lim_{t \to 0} \frac{P_t f(x)-f(x)}{t}, \qquad f \in C_c^{\infty}(\mbb{R}^d), \, x \in \mbb{R}^d,
\end{equation*}
which means that $A=-\psi(D)$ is the infinitesimal generator of $(X_t)_{t \geq 0}$. The L\'evy process $(X_t)_{t \geq 0}$ is uniquely determined by $\psi$, the so-called characteristic exponent of $(X_t)_{t \geq 0}$, and by the associated L\'evy triplet $(b,Q,\nu)$. The following theorem is our main result.

\begin{thm} \label{main-3}
	Let $(X_t)_{t \geq 0}$ be a L\'evy process with L\'evy triplet $(b,Q,\nu)$ and characteristic exponent $\psi$, and denote by $Af=-\psi(D)f$ the associated L\'evy generator. Assume that 
	\begin{enumerate}[label*=\upshape (C\arabic*),ref=\upshape C\arabic*] 
		\item\label{C1} $X_t$ has for each $t>0$ a density $p_t \in C_b^1(\mbb{R}^d)$ with respect to Lebesgue measure,
		\item\label{C2} there exists some $\beta>0$ such that $\int_{|y| \geq 1} |y|^{\beta} \, \nu(dy)<\infty$.
	\end{enumerate}
	If $u: \mbb{R}^d\to \mbb{R}$ is a weak solution to \begin{equation*}
		Au=0 \quad \text{in $\mbb{R}^d$}
	\end{equation*}
	satisfying $|u(x)| \leq M(1+|x|^{\gamma})$, $x \in \mbb{R}^d$, for some $M>0$ and $\gamma \in [0,\beta)$, then $u$ is a polynomial of degree at most $\floor{\gamma}$. In particular, $A$ has the Liouville property \eqref{liou-p}.
\end{thm}

\begin{bem} \label{main-5} \begin{enumerate}[wide, labelwidth=!, labelindent=0pt]
	\item Weak solutions to $Au=0$ are only determined up to a Lebesgue null set, cf.\ Section~\ref{weak}. When we write ``$u$ is a polynomial'', this means that $u$ has a representative which is a polynomial, i.e.\ there is a polynomial $\tilde{u}$ such that $u=\tilde{u}$ Lebesgue almost everywhere.
	\item If $(X_t)_{t \geq 0}$ is a  Brownian motion, then \eqref{C1} is trivial and \eqref{C2} holds for all $\beta>0$; consequently, we recover the classical Liouville theorem for the Laplacian. 
	\item A sufficient condition for \eqref{C1} is the Hartman--Wintner condition, \begin{equation*}
		\lim_{|\xi| \to \infty} \frac{\re \psi(\xi)}{\log(|\xi|)}=\infty,
	\end{equation*}
	see \cite{knop13} for a thorough discussion. 
	\item Condition~\eqref{C2} is equivalent to assuming that $\mbb{E}(|X_t|^{\beta})=\int_{\mbb{R}^d} |x|^{\beta} p_t(x) \, dx$ is finite for some (all) $t>0$, cf.\ \cite{sato}. Consequently, \eqref{C2} implies, in particular, that $P_t u(x)=\mbb{E}u(x+X_t)$ is well defined for any measurable function $u$ satisfying the growth condition $|u(x)| \leq M(1+|x|^{\beta})$.
	\item The conditions \eqref{C1} and \eqref{C2} are quite mild assumptions, which hold for a large class of pseudo-differential operators. The recent paper \cite{knop19} does, however, indicate that our conditions are not sharp; it is shown that $A=-\psi(D)$ has the Liouville property \eqref{liou-p} iff $\{\psi=0\}=\{0\}$. By the Riemann--Lebesgue lemma, \eqref{C1} implies $\{\psi=0\}=\{0\}$ but the converse is not true.
\end{enumerate} \end{bem}

Let us sketch the idea of the proof of Theorem~\ref{main-3}. First, we show under mild assumptions that every weak solution to the equation $Au=0$ gives rise to a (continuous) solution to the convolution equation $P_t u=u$. The intuition behind this result comes from Dynkin's formula: If $Au=0$ and $u$ is, say, twice differentiable and bounded, then Dynkin's formula, cf.\  \cite[Lemma 4.1.14]{jacob1}, shows \begin{equation*}
	P_t u - u = \int_0^t P_s Au \, ds = 0 \fa t \geq 0.
\end{equation*}
Secondly, we use that the convolution operator $P_t$ has smoothing properties, i.e.\ $P_t u$ has a higher regularity than $u$. If $u$ is a solution to $Au=0$, and hence to $P_t u=u$, then these regularizing properties of $P_t$ allow us to establish suitable H\"older estimates for $u$ which lead, by iteration, to the conclusion that $u$ is smooth; thus a polynomial. \par \medskip

The remaining article is structured as follows. In Section~\ref{weak} we introduce the notion of weak solutions and study the connection between the ``Laplace'' equation $Au=0$ and the convolution equation $P_t u=u$. In Section~\ref{reg} we establish regularity estimates for the semigroup $(P_t)_{t \geq 0}$, which are of independent interest. The Liouville theorem is proved in Section~\ref{proof}.

\section{Weak solutions} \label{weak}

Let $A=-\psi(D)$ be a pseudo-differential operator with continuous negative definite symbol $\psi:\mbb{R}^d \to \mbb{C}$, cf.\ \eqref{char-exp}. Since $\overline{\psi(\xi)}=\psi(-\xi)$ for all $\xi \in \mbb{R}^d$, an application of Plancherel's theorem shows that the pseudo-differential operator $A^* f := - \overline{\psi}(D)f$ is the adjoint of $A$ in $L^2(dx)$. Indeed, if $\varphi,f \in C_c^{\infty}(\mbb{R}^d)$, then \begin{align*}
	\sprod{Af,\varphi}_{L^2}
	= \sprod{\widehat{Af},\check{\varphi}}_{L^2}
	= \sprod{- \psi \hat{f},\check{\varphi}}_{L^2}
	= \sprod{\hat{f}, -\psi \check{\varphi}}_{L^2}
	= \sprod{\hat{f},\widecheck{A^* \varphi}}_{L^2}
	= \sprod{f,A^* \varphi}_{L^2},
\end{align*}
where $\check{\varphi}$ denotes the inverse Fourier transform of $\varphi$.

\begin{defn}
	Let $A$ be a pseudo-differential operator with continuous negative definite symbol $\psi: \mbb{R}^d \to\mbb{C}$. Let $U \subseteq \mbb{R}^d$ be open and $f \in L^1_{\loc}(U)$. A measurable function $u: \mbb{R}^d \to \mbb{R}$ is a weak solution to \begin{equation*}
		Au=f \quad \text{in $U$}
	\end{equation*}
	if \begin{equation}
		\forall \varphi \in C_c^{\infty}(U)\::\: \int_{\mbb{R}^d} u(x) A^* \varphi(x) \, dx = \int_U f(x) \varphi(x) \, dx. \label{eq-weak}
	\end{equation}
\end{defn}

In \eqref{eq-weak} we implicitly assume that the integrals exist. For the integral on the right-hand side, the existence is evident from $\varphi \in C_c^{\infty}(U)$ and $f \in L_{\loc}^1(U)$. The other integral is harder to deal with because $A^*$ is a non-local operator, i.e.\ decay properties of $\varphi$ (e.g.\ compactness of the support) do not carry over to $A^* \varphi$. Our first result in this section shows that the decay of $A^* \varphi$ is closely linked to the existence of fractional moments $\int_{|y| \geq 1} |y|^{\beta} \, \nu(dy)$ of the L\'evy measure $\nu$, associated with $\psi$ via \eqref{char-exp}; see \cite[Lemma 2.1]{fall16} for a related result.

\begin{prop} \label{p-1}
	Let $\psi:\mbb{R}^d \to \mbb{C}$ be a continuous negative definite function with triplet $(b,Q,\nu)$. If $\beta>0$ is such that $\int_{|y| \geq 1} |y|^{\beta} \, \nu(dy)<\infty$, then the pseudo-differential operator $A=-\psi(D)$ satisfies \begin{equation}
		\int_{\mbb{R}^d} (1+|x|^{\beta}) |A\varphi(x)| \, dx < \infty \fa  \varphi \in C_c^{\infty}(\mbb{R}^d). \label{p-eq6}
	\end{equation}
	More precisely, there exists for all $R>0$ a constant $C>0$ such that every $\varphi \in C_c^{\infty}(\mbb{R}^d)$ with $\spt \varphi \subset B(0,R)$ satisfies \begin{equation}
		\int_{\mbb{R}^d} (1+|x|^{\beta}) |A\varphi(x)| \, dx \leq C \|\varphi\|_{C_b^2(\mbb{R}^d)} \left( |b| + |Q| + \int_{|y| \leq 1} |y|^2 \, \nu(dy) + \int_{|y|>1} |y|^{\beta} \, \nu(dy) \right). \label{p-eq7}
	\end{equation}
\end{prop}
Let us mention that $\int_{|y| \geq 1} |y|^{\beta} \, \nu(dy) < \infty$ is actually equivalent to \eqref{p-eq6}. Here, we need (and prove) only sufficiency for \eqref{p-eq6}; for the converse implication see \cite[Theorem 4.1]{hulanicki}. \par
Proposition~\ref{p-1} gives a sufficient condition such the integral on the left-hand side of \eqref{eq-weak} exists: Since the adjoint $A^*$ is a pseudo-differential operator with symbol $\overline{\psi}$ and triplet $(-b,Q,\nu(-\cdot))$, Proposition~\ref{p-1} shows that $\int_{\mbb{R}^d} |u(x)|\, |A^* \varphi(x)| \, dx$ is finite for every measurable function $u$ satisfying $|u(x)| \leq M(1+|x|^{\beta})$ for some $\beta \geq 0$ with $\int_{|y| \geq 1} |y|^{\beta} \, \nu(dy)<\infty$. 

\begin{proof}[Proof of Proposition~\ref{p-1}]
	Since the assertion is obvious for the local part of $A$, we may assume without loss of generality that $b=0$ and $Q=0$. Fix $\varphi \in C_c^{\infty}(\mbb{R}^d)$ with $\spt \varphi \subset B(0,R)$. For $x \in \mbb{R}^d$ with $|x| \geq 2R$, we have
	\begin{align*}
		|x|^{\beta} |A\varphi(x)|
		\leq |x|^{\beta} \int_{|y+x|<R} |\varphi(x+y)| \, \nu(dy) 
		&\leq \int_{|y| \geq |x|-R} |\varphi(x+y)| \frac{|x|^{\beta}}{(|x|-R)^{\beta}} |y|^{\beta} \, \nu(dy) \\
		&\leq C \int_{|y| \geq R} |\varphi(x+y)| \, |y|^{\beta} \, \nu(dy)
	\end{align*}
	for some constant $C=C(R)$. Integrating with respect to $x$, we find by Tonelli's theorem that
	\begin{align*}
		\int_{|x| \geq 2R} |x|^{\beta} |A \varphi(x)| \, dx
		\leq C \|\varphi\|_{\infty} (2R)^d \int_{|y| \geq R} |y|^{\beta} \, \nu(dy).
	\end{align*}
	On the other hand, it is immediate from Taylor's formula that \begin{equation*}
		\|A\varphi\|_{\infty} \leq 2 \|\varphi\|_{C_b^2(\mbb{R})} \int_{y \neq 0} \min\{1,|y|^2\} \, \nu(dy),
	\end{equation*}
	and this yields the required estimate for $\int_{|x| <2R} (1+|x|^{\beta}) |A\varphi(x)| \, dx$.
\end{proof}

Next we establish a connection between the ``Laplace'' equation $Au=0$ and the convolution equation $P_t u=u$. 

\begin{lem} \label{p-3}
	Let $(X_t)_{t \geq 0}$ be a L\'evy process with L\'evy triplet $(b,Q,\nu)$, infinitesimal generator $(A,\mc{D}(A))$ and semigroup $(P_t)_{t \geq 0}$. Assume that $X_t$ has for $t>0$ a density $p_t \in C_b(\mbb{R}^d)$ with respect to Lebesgue measure, and let $\beta \geq 0$ be such that $\int_{|y| \geq 1} |y|^{\beta} \, \nu(dy)<\infty$. 
	If $u:\mbb{R}^d \to \mbb{R}$ is a measurable function with $|u(x)| \leq M(1+|x|^{\beta})$, $x \in \mbb{R}^d$, solving \begin{equation*}
		Au=0 \quad \text{weakly in $\mbb{R}^d$},
	\end{equation*}
	then there exists $\tilde{u} \in C(\mbb{R}^d)$ such that $u=\bar{u}$ Lebesgue almost everywhere and $\tilde{u} = P_t \tilde{u}$  for all $t>0$.
\end{lem}

Note that the exceptional null set $\{\tilde{u} \neq P_t \tilde{u}\}$ does, in general, depend on $t$; for the application which we have in mind, that is, for the proof of Liouville's theorem, this is not a problem since we will use the result only for $t=1$.

\begin{proof}
	Take $\varphi \in C_c^{\infty}(\mbb{R}^d)$ such that $\varphi \geq 0$ and $\int_{\mbb{R}^d} \varphi(x) \, dx=1$. Set $\varphi_{\eps}(x):=\eps^{-d} \varphi(x/\eps)$ and \begin{equation*}
		u_{\eps}(x):=(u \ast \varphi_{\eps})(x) := \int_{\mbb{R}^d} u(x-y) \varphi_{\eps}(y) \, dy, \qquad x \in \mbb{R}^d,
	\end{equation*}
	for $\eps>0$.  Using \begin{equation*}
		(a+b)^{\beta} \leq c_{\beta} (a^{\beta}+b^{\beta}), \qquad a,b \geq 0,
	\end{equation*}
	it follows that \begin{align}
		|u_{\eps}(x)|
		\leq M \int_{\mbb{R}^d} (1+|x-y|^{\beta}) |\varphi_{\eps}(y) \, dy 
		&\leq M c_{\beta} (1+|x|^{\beta}) \left(\int_{\mbb{R}^d} |\varphi_{\eps}(y)| \, dy + \int_{\mbb{R}^d} |y|^{\beta} |\varphi_{\eps}(y)| \, dy \right) \notag \\
		&\leq C_1 (1+|x|^{\beta}) \label{p-st3}
	\end{align}
	for some constant $C_1>0$ which does not depend on $\eps$, $x$ and $u$. 
	As $\int_{|y| \geq 1} |y|^{\beta} \, \nu(dy)<\infty$, the L\'evy process has fractional moments of order $\beta$, i.e. $\mbb{E}(|X_t|^{\beta})=\int |y|^{\beta} p_t(y) \, dy<\infty$, see e.g.\ \cite[Theorem 25.3]{sato} or \cite[Theorem 4.1]{moments}, and so $P_t u$ and $P_t u_{\eps}$ are well-defined. We have \begin{align*}
		|P_t u(x)-P_t u_{\eps}(x)| \leq \|p_t\|_{\infty} \int_{|y| \leq R} |u(y)-u_{\eps}(y)| \, dy + 2M \int_{|y|>R} (1+|y|^{\beta}) p_t(y-x) \, dy.
	\end{align*}
	For fixed $x \in \mbb{R}^d$, it follows from the dominated convergence theorem that the second term on the right-hand side is less than, say, $\varrho>0$, for $R$ large enough. Since $u_{\eps} \to u$ in $L^1_{\loc}(dx)$, the first term is less than $\varrho$ for small $\eps>0$. Hence, $P_t u_{\eps}(x) \to P_t u(x)$ as $\eps \to 0$ for each $x \in \mbb{R}^d$. Next we show that \begin{equation}
		P_t u_{\eps}(x) = u_{\eps}(x) \fa t>0, \, x \in \mbb{R}^d,\, \eps>0. \label{p-eq31}
	\end{equation}
	By the definition of $P_t u$ and $u_{\eps}$,we have \begin{align*}
		P_t u_{\eps}(x) = \int_{\mbb{R}^d} \left( \int_{\mbb{R}^d} u(z) \varphi_{\eps}(y-z) \, dz\right)   p_t(y-x) \, dy.
	\end{align*}
	Because of the growth estimate in \eqref{p-st3}, we may apply Fubini's theorem: \begin{align*}
		P_t u_{\eps}(x)
		&= \int_{\mbb{R}^d} \left( \int_{\mbb{R}^d} \varphi_{\eps}(y-z) p_t(y-x) \, dy \right) u(z) \, dz \\
		&= u_{\eps}(x) + \int_{\mbb{R}^d} u(z) \left(\mbb{E}\varphi_{\eps}(x-z+X_t) - \varphi_{\eps}(x-z)\right) \, dz
		=: u_{\eps}(x) + \Delta.
	\end{align*}
	It remains to show that $\Delta=0$. As $\varphi_{\eps} \in C_c^{\infty}(\mbb{R}^d)$, an application of Dynkin's formula gives \begin{align*}
		\Delta = \int_{\mbb{R}^d} u(z) \int_0^t \mbb{E}((A\varphi_{\eps})(x-z+X_s)) \, ds \, dz.
	\end{align*}
	Applying Lemma~\ref{p-1}, using the growth condition on $u$ and the fact that $\int_0^t \mbb{E}(|X_s|^{\beta}) \, ds <\infty$, cf.\ \cite[Theorem 25.18]{sato} or \cite[Theorem 4.1]{moments}, we find that \begin{equation*}
		\mbb{E}\left( \int_0^t  \int_{\mbb{R}^d} |u(z+X_s)| \, |(A\varphi_{\eps})(x-z)| \, dz \, ds \right)< \infty,
	\end{equation*}
	and therefore we may apply once more Fubini's theorem: \begin{align*}
		\Delta 
		= \mbb{E}\left(\int_0^t \int_{\mbb{R}^d} (A\varphi_{\eps})(x-z+X_s) u(z) \, dz \, ds \right).
	\end{align*}
	From \begin{equation*}
		(A \phi)(y-z) = (A\phi(\bullet+y))(-z) \quad \text{and} \quad (A\phi)(-z), = (A^* \phi(-\bullet))(z).
	\end{equation*}
	we conclude that \begin{equation*}
		\Delta = \mbb{E} \left( \int_0^t \int_{\mbb{R}^d} (A^* \varphi_{\eps}(x+X_s-\bullet))(z) u(z) \, dz \,ds \right).
	\end{equation*}
	Since $z \mapsto \varphi_{\eps}(x+X_s(\omega)-z) \in C_c^{\infty}(\mbb{R}^d)$ for each fixed $\omega \in \Omega$, $s \in [0,t]$ and $x \in \mbb{R}^d$, it follows from $Au=0$ weakly that the inner integral on the right-hand side is zero, and so $\Delta=0$. This finishes the proof of \eqref{p-eq31}. As $u_{\eps} \to u$ in $L_1^{\loc}$, there exists a subsequence converging Lebesgue almost everywhere. Letting $\eps \to 0$ in \eqref{p-eq31} along this subsequence, we get $P_t u=u$ Lebesgue almost everywhere. If we set $\tilde{u} := P_1 u$, then $u=P_1 u = \tilde{u}$ Lebesgue almost everywhere and \begin{equation*}
		\tilde{u} = u = P_t u = P_t \tilde{u} \quad \text{a.e.}
	\end{equation*}
	where the latter equality follows from the fact that $P_t$ does not see Lebesgue null sets since $X_t$ has a density with respect to Lebesgue measure. Finally, we note that $\tilde{u} \in C(\mbb{R}^d)$. Indeed, given $\eps>0$ and $r>0$, there is some $R>r$ such that \begin{equation*}
		\sup_{x \in B(0,r)} \int_{|y| \geq R}(1+|y|^{\beta}) p_1(y-x) \, dy= \sup_{x \in B(0,r)} \int_{|y| \geq R}(1+|y+x|^{\beta}) p_1(y) \, dy  \leq \epsilon.
	\end{equation*}
	Hence, for all $x,z \in B(0,r)$ \begin{align*}
		|\tilde{u}(x)-\tilde{u}(z)|
		&\leq \int_{|y| \leq R} |u(y)| \, |p_1(y-x)-p_1(y-z)| \, dy + \int_{|y| \geq R} |u(y)| \, |p_1(y-x)-p_1(y-z)| \, dy \\
		&\leq M(1+R^{\beta}) R^d \sup_{\substack{|u-v| \leq |x-z| \\ u,v \in B(0,2R)}} |p_1(u)-p_1(v)| + 2M \epsilon 
		\xrightarrow[]{|x-z| \to 0} 2M \epsilon \xrightarrow[]{\eps \to 0} 0,
	\end{align*}
	i.e. $\tilde{u}$ is continuous. Since $\tilde{u}$ and $P_t \tilde{u}$ are continuous, it follows from $\tilde{u} = P_t \tilde{u}$ Lebesgue almost everywhere that $\tilde{u}(x) = P_t \tilde{u}(x)$ for all $x \in \mbb{R}^d$. 
\end{proof}

\section{Regularity estimates for semigroups associated with L\'evy processes} \label{reg}

Let $(X_t)_{t \geq 0}$ be a L\'evy process with transition density $p_t$, $t>0$, and semigroup \begin{equation*}
	P_t u(x) := \mbb{E}u(x+X_t) = \int_{\mbb{R}^d} u(x+y) p_t(y) \, dy, \qquad t>0, \, x \in \mbb{R}^d.
\end{equation*}
If $u:\mbb{R}^d \to \mbb{R}$ is bounded and Borel measurable, then $P_t u$ is continuous, being convolution of a bounded function with an integrable function, cf.\ \cite[Theorem 15.8]{mims}. In this section, we study the regularity of $x \mapsto P_t u(x)$ for \emph{unbounded} functions $u$. If $u$ is unbounded, then we need some assumptions to make sense of the integral appearing in the definition of $P_t u$. It is natural to assume that there exists a constant $\beta>0$ such that the associated L\'evy measure $\nu$ satisfies $\int_{|y| \geq 1} |y|^{\beta} \, \nu(dy)<\infty$. This condition ensures that $\mbb{E}(|X_t|^{\beta})<\infty$ for all $t \geq 0$, cf.\ Sato \cite{sato}, and so $P_t u$ is well-defined for any function $u$ satisfying $|u(x)| \leq M(1+|x|^{\beta})$, $x \in \mbb{R}^d$, for some $M>0$. Under the assumption that $p_t \in C_b^1(\mbb{R}^d)$, we will show that $P_t u$ is locally H\"older continuous for every function $u$ satisfying $|u(x)| \leq M(1+|x|^{\gamma})$, $x \in \mbb{R}^d$, for some $\gamma<\beta$. Before stating the result, let us give a word of caution.  As \begin{equation*}
	P_t u(x) = \int_{\mbb{R}^d} u(y) p_t(y-x) \, dy,
\end{equation*}
a naive differentiation yields \begin{equation*}
	\nabla P_t u(x) =- \int_{\mbb{R}^d} u(y) \nabla p_t(y-x) \, dy,
\end{equation*}
and therefore one might suspect that $P_t u$ is differentiable (and not only locally H\"older continuous). In general, it is not possible to make this calculation rigorous, even if $u$ is bounded. To start with, it is not clear that the integral $\int_{\mbb{R}^d}|u(y)| \, |\nabla p_t(y-x)| \, dy$ is finite since the decay of $p_t$ does not necessarily carry over to its derivatives. However, there is an interesting -- and wide -- class of L\'evy processes for which the above reasoning can be made rigorous, and we will work out the details in the second part of this section.

\begin{lem} \label{p-5}
	Let $(X_t)_{t \geq 0}$ be a L\'evy process with L\'evy triplet $(b,Q,\nu)$ and semigroup $(P_t)_{t \geq 0}$. Let $\beta>0$ be such that $\int_{|y| \geq 1} |y|^{\beta} \, \nu(dy)<\infty$, and assume that $X_t$ has for some $t>0$ a density $p_t \in C_b^1(\mbb{R}^d)$ with respect to Lebesgue measure. If $u$ is a measurable function satisfying $|u(x)| \leq M (1+|x|^{\gamma})$ for some $M>0$ and $\gamma \in [0,\beta)$, then \begin{equation}
		|P_t u(rx+rh)-P_t u(rx)| \leq C M r^{\gamma} |h|^{\varrho}, \qquad |x|,|h| \leq 1, \, r \geq 1, \label{p-eq51}
	\end{equation}
	where $\varrho:=\frac{\beta-\gamma}{d+\beta} \in (0,1)$ and $C=C(t,\beta)<\infty$ is a constant which does not depend on $u$. In particular, $x \mapsto P_t u(x)$ is H\"older continuous of order $\varrho$ on any compact set $K \subseteq \mbb{R}^d$ and \begin{equation*}
		\|P_t u\|_{C_b^{\varrho}(B(0,r))} \leq (C+2) M r^{\gamma} \fa r \geq 1.
	\end{equation*}
\end{lem}

\begin{proof}
	Because of the growth assumption on $u$, it follows from $\mbb{E}(|X_t|^{\beta})<\infty$ that $P_t u$ is well-defined. Fix $r,R \geq 1$ and $x,h\in \mbb{R}^d$ with $|h|, |x| \leq 1$. By the definition of the semigroup, \begin{align*}
		\Delta_h
		:= P_t u(rx+rh)-P_t u(rx)
		&= \int_{\mbb{R}^d} u(y) \left( p_t(y+rx+rh)-p_t(y+rx)\right) \, dy \\
		&= r^{-d} \int_{\mbb{R}^d} u(rz) \left( p_t(rz+rx+rh)-p_t(rz+rx) \right) \, dz.
	\end{align*} 
	Thus, $\Delta_h=\Delta_h^1 + \Delta_h^2$, where \begin{align*}
		\Delta_h^1 &:= r^{-d} \int_{|z| \leq R}u(rz) \left( p_t(rz+rx+rh)-p_t(rz+rx) \right) \, dz, \\
		\Delta_h^2 &:=  r^{-d} \int_{|z| > R}u(rz) \left( p_t(rz+rx+rh)-p_t(rz+rx) \right) \, dz.
	\end{align*}
	Applying the mean value theorem and using the growth condition on $u$, we find that \begin{align*}
		|\Delta_h^1|
		\leq |h| r^{-d+1} \|\nabla p_t\|_{\infty} \int_{|z| \leq R} |u(rz)|\, dz
		\leq 2M |h| r^{-d+1+\gamma} \|\nabla p_t\|_{\infty} R^{d+\gamma}.
	\end{align*}
	For the second term, we use again the growth condition on $u$: \begin{align*}
		|\Delta_h^2|
		&\leq 4M r^{-d+\gamma} \sup_{|h| \leq 1} \int_{|z|>R} |z|^{\gamma} p_t(rz+rx+rh) \, dz \\
		&\leq 4M r^{-d+\gamma} R^{\gamma-\beta} \sup_{|h| \leq 1} \int_{\mbb{R}^d} |z|^{\beta} p_t(rz+rh+rx) \, dz.
	\end{align*}
	Performing a change of variables and using the elementary estimate \begin{equation*}
		(a+b)^{\beta} \leq c_{\beta} (a^{\beta}+b^{\beta}), \qquad a,b \geq 0,
	\end{equation*}
	we get \begin{align*}
		|\Delta_h^2|
		&\leq 4M r^{\gamma-\beta} R^{\gamma-\beta} \sup_{|h| \leq 1} \int_{\mbb{R}^d} |y-(rx+rh)|^{\beta} p_t(y) \, dy \\
		&\leq 4 \, 2^{\beta} c_{\beta} M r^{\gamma} R^{\gamma-\beta} \left(1+ \int_{\mbb{R}^d} |y|^{\beta} p_t(y) \, dy \right).
	\end{align*}
	Note that the integral on the right-hand side is finite since $\mbb{E}(|X_t|^{\beta})<\infty$. Consequently, we have shown that there exists a constant $C=C(\beta,t)>0$ such that \begin{equation*}
		|P_t u(rx+rh)-P_t u(rx)| = |\Delta_h| \leq C M r^{\gamma} R^{d+\gamma} |h| + C M r^{\gamma} R^{\gamma-\beta} \fa |h|,|x| \leq 1, \, r \geq 1.
	\end{equation*}
	Choosing $R:=|h|^{-1/(d+\beta)}$ gives \eqref{p-eq51}. The remaining assertion is obvious from \eqref{p-eq51}.
\end{proof}

If $(P_t)_{t \geq 0}$ is the semigroup associated with a subordinated Brownian motion $(X_t)_{t \geq 0}$, then the regularity estimate from Proposition~\ref{p-7} can be improved. We do not need this strengthened version for the proof of the Liouville theorem, but we present the proof since we believe that the result is of independent interest. Recall that a L\'evy process $(S_t)_{t \geq 0}$ is a subordinator if $(S_t)_{t \geq 0}$ has non-decreasing sample paths.

\begin{prop} \label{p-7}
	Let $(X_t)_{t \geq 0}$ be a L\'evy process which is of the form $X_t = B_{S_t}$ for a $d$-dimensional Brownian motion $(B_t)_{t \geq 0}$ and a subordinator $(S_t)_{t \geq 0}$ satisfying $\mathbb{P}(S_t=0)=0$ for all $t>0$. Denote by $(b,Q,\nu)$ the L\'evy triplet of $(X_t)_{t \geq 0}$, and let $\beta>0$ be such that $\int_{|y| \geq 1} |y|^{\beta}\, \nu(dy)<\infty$. If $u:\mbb{R}^d \to \mbb{R}$ is a measurable function satisfying $|u(x)| \leq M(1+|x|^{\gamma})$, $x \in \mbb{R}^d$, for some $M>0$ and $\gamma \in [0,\beta]$, then $x \mapsto P_t u(x)$ is smooth for all $t>0$ and \begin{equation}
		\|P_t u\|_{C_b^k(B(0,r))} \leq C_k M r^{\gamma} \fa r \geq 1, \, k \geq 1, \label{p-eq71}
	\end{equation}
	where $C_k=C_k(t)$ is a finite constant, which does not depend on $u$ and $r$.
\end{prop}

Let us mention that $\mathbb{P}(S_t=0)=0$ is equivalent to assuming that $(X_t)_{t \geq 0}$ has a density with respect to Lebesgue measure, cf.\ \cite[Lemma 4.6]{schoenberg}.

\begin{proof}[Proof of Proposition~\ref{p-7}]
	For $k \geq 1$ let $(B_t^{(k)})_{t \geq 0}$ be a $k$-dimensional Brownian motion. The process $X_t^{(k)} := B_{S_t}^{(k)}$ is a L\'evy process with L\'evy triplet, say, $(b^{(k)},Q^{(k)},\nu^{(k)})$, cf.\ \cite{bernstein} or \cite{sato}. By definition, $X_t=X_t^{(d)}$ and $\nu=\nu^{(d)}$. Since $(B_t^{(k)})_{t \geq 0}$ and $(S_t)_{t \geq 0}$ are independent, cf.\ \cite[Theorem II.6.3]{ikeda}, it follows from $B_t^{(k)}= \sqrt{t} B_1^{(k)}$ in distribution that \begin{equation*}
		\mathbb{E}(|B_{S_t}^{(k)}|^{\beta}) = \mathbb{E}(|S_t|^{\beta/2}) \mathbb{E}(|B_1^{(k)}|^{\beta}).
	\end{equation*}
	Consequently, \begin{align*}
		\int_{|y| \geq 1} |y|^{\beta} \, \nu^{(k)}(dy) < \infty
		\iff \mathbb{E}(|B_{S_t}^{(k)}|^{\beta})< \infty 
		\iff \mbb{E}(|S_t|^{\beta/2})<\infty,
	\end{align*}
	and so the finiteness of the fractional moment $\int_{|y| \geq 1} |y|^{\beta} \, \nu^{(k)}(dy)$ does not depend on the dimension $k$. By assumption, the moment is finite for $k=d$, and hence it is finite for all $k \geq 1$. Thus, $\mbb{E}(|X_t^{(k)}|^{\beta})<\infty$ for all $k \geq 1$ and $t \geq 0$. As $\mathbb{P}(S_t=0)$, the process $(X_t^{(k)})_{t \geq 0}$ has a rotational invariant and smooth density $p_t^{(k)}(x)=p_t^{(k)}(|x|)$, \begin{equation*}
		\mbb{P}(X_t^{(k)} \in dx) = p_t^{(k)}(|x|) \, dx
	\end{equation*}
	 and \begin{equation}
		\frac{d}{dr} p_t^{(k)}(r) = - 2\pi p_t^{(k+2)}(r), \qquad k \geq 1, \, r>0, \label{p-eq75}
	\end{equation}
	cf.\  \cite[Corollary 3.2, Lemma 4.6]{schoenberg}. Using polar coordinates, we get \begin{equation*}
		\int_{|x| \geq 1} |x|^{\gamma} |\nabla p_t^{(k)}(x)| \, dx
		= c \int_{r \geq 1} r^{\gamma+d} p_t^{(k+2)}(r) \, dr
		\leq c' \mbb{E}(|X_t^{(k+2)}|^{\gamma}) <\infty
	\end{equation*}
	for all $\gamma \in [0,\beta]$ and $k \geq 1$. Since the continuous function $|\nabla p_t^{(k)}|$ is bounded on compact sets, this implies \begin{equation*}
		\int_{\mbb{R}^k} (1+|x|^{\gamma}) |\nabla p_t^{(k)}(x)| \, dx< \infty \fa t \geq 0, \, \gamma \in [0,\beta], \, k \geq 1.
	\end{equation*}
	Applying iteratively \eqref{p-eq75} with $k=d+2n$, $n \in \mbb{N}$, we find that \begin{equation}
		\int_{\mbb{R}^d} (1+|x|^{\gamma}) |\partial^{\alpha} p_t^{(d)}(x)| \, dx < \infty \label{p-eq77}
	\end{equation}
	for all $\gamma \in [0,\beta]$, $t \geq 0$ and all multi-indices $\alpha \in \mbb{N}_0^d$. Now we return to our original problem, i.e.\ we study the regularity of the semigroup $(P_t)_{t \geq 0}$ associated with $X_t= X_t^{(d)}$. Fix a measurable function $u$ with $|u(x)| \leq M(1+|x|^{\gamma})$ for some constants $M>0$ and $\gamma \in [0,\beta]$. By definition, \begin{equation*}
		  P_t u(x) = \mbb{E}u(x+X_t) = \int u(y) p_t^{(d)}(y-x) \, dy, \quad x \in \mbb{R}^d.
	\end{equation*}
	By \eqref{p-eq77}, we have $\int_K \int_{\mbb{R}^d} |u(y)|\, |\partial_{x_j} p_t^{(d)}(y-x)| \, dy \, dx < \infty$ for every $j=1,\ldots,d$ and every compact set $K \subseteq \mbb{R}^d$. Moreover, it follows by a similar reasoning to that at the end of the proof of Lemma~\ref{p-3} that the mapping \begin{equation*}
		x \mapsto \int_{\mbb{R}^d} u(y) \partial_{x_j} p_t^{(d)}(y-x) \, dy
	\end{equation*}
	is continuous. Applying the differentiation lemma for parametrized integrals, cf.\ \cite[Proposition A.1]{euler-maruyama}, we obtain that \begin{equation*}
		\partial_{x_j} P_t u(x) = -\int_{\mbb{R}^d} u(y) \partial_{x_j} p_t^{(d)}(y-x) \, dy, \qquad j=1,\ldots,d, \, x \in \mbb{R}^d.
	\end{equation*}
	Performing a change of variables $y \rightsquigarrow y+x$, it is immediate from \eqref{p-eq77} and the growth condition on $u$ that $\|P_t u\|_{C_b^1(B(0,R))} \leq C M R^{\beta}$,  $R \geq 1$, for some constant $C>0$. Iterating the procedure proves the assertion for higher order derivatives.
\end{proof}

\section{Proof of Liouville's theorem} \label{proof}

In this section, we prove the Liouville theorem, cf.\ Theorem~\ref{main-3}. First, we use a general result by Choquet \& Deny \cite{choquet60} to show that the only bounded solutions to the convolution equation $P_t u=u$ are the trivial ones. 

\begin{prop} \label{p-9}
	Let $(X_t)_{t \geq 0}$ be a L\'evy process with characteristic exponent $\psi$ and semigroup $(P_t)_{t \geq 0}$, and denote by $Af = - \psi(D) f$ the associated L\'evy generator. Assume that $X_t$ has a density $p_t \in C_b(\mbb{R}^d)$ for some $t>0$. \begin{enumerate}
		\item\label{p-9-i} If $u$ is a bounded measurable function such that $P_t u=u$ a.e., then $u$ is constant a.e.
		\item\label{p-9-ii} (Liouville property) If $u \in L^{\infty}(\mbb{R}^d)$ and $Au=0$ weakly, then $u$ is constant a.e.
	\end{enumerate}
\end{prop}

\begin{proof} \begin{enumerate}[wide, labelwidth=!, labelindent=0pt]
	\item  Without loss of generality, we may assume that $P_t u(x)=u(x)$ for all $x \in \mbb{R}^d$; otherwise replace $u$ by $\tilde{u} :=P_t u$ and note that $P_t u=P_t \tilde{u}$ as $X_t$ has a density with respect to Lebesgue measure. Since $\int_{\mbb{R}^d} p_t(y) \, dy =1$ and $p_t \geq 0$ is continuous, there exist $x_0 \in \mbb{R}^d$ and $r>0$ such that $p_t(y) >0$ for all $y \in B(x_0,r)$. In particular, $B(x_0,r)$ is contained in the support of the distribution of $X_t$. By \cite[Theorem 1]{choquet60}, this implies \begin{equation*}
		u(x)=u(x+y) \fa x \in \mbb{R}^d, \, y \in B(x_0,r),
	\end{equation*}
	Hence, $u$ is constant. 
	\item This is immediate from Lemma~\ref{p-3} and \eqref{p-9-i}. \qedhere
\end{enumerate} \end{proof}

We are now ready to prove the Liouville theorem. 

\begin{proof}[Proof of Theorem~\ref{main-3}]
	By Lemma~\ref{p-3}, we may assume without loss of generality that $u$ is continuous and $u(x)=P_1 u(x)$ for all $x \in \mbb{R}^d$. Applying Lemma~\ref{p-5}, we find that there exists a constant $C>0$ such that \begin{equation}
		|u(rx'+rh')-u(rx')|
		= |P_1 u(rx'+rh')-P_1 u(rx')|
		\leq C M r^{\gamma} |h'|^{\varrho}, \qquad |h'|,|x'| \leq 1, \, r \geq 1 \label{p-eq91}
	\end{equation}
	for $\varrho:=(\beta-\gamma)/(d+\beta)>0$ and some constant $C=C(\beta)>0$. This implies \begin{equation*}
		|u(x+h)-u(x)| \leq 2C M (1+|x|^{\gamma-\varrho}) |h|^{\varrho} \fa x \in \mbb{R}^d, |h| \leq 1.
	\end{equation*}
	Indeed: If $|x| \leq 1$, then this follows from \eqref{p-eq91} for $r=1$, $x'=x$ and $h'=h$; if $|x|>1$ we choose $r=|x|$, $h'=h/r$ and $x'=x/r$ in \eqref{p-eq91}. This means that for each fixed $h \in \mbb{R}^d$, $0<|h| \leq 1$, the function $v(x) := |h|^{-\varrho} (u(x+h)-u(x))$ satisfies \begin{equation*}
		|v(x)| \leq 2CM \left(1+|x|^{\gamma-\varrho}\right), \qquad x \in \mbb{R}^d.
	\end{equation*}
	Since the semigroup $(P_t)_{t \geq 0}$ is invariant under translations, we have $P_1 v=v$, and therefore we can apply the above reasoning to $v$ (instead of $u$) to obtain that \begin{equation*}
		|v(x+h)-v(x)| \leq 4C^2 M^2 \left(1+|x|^{\gamma-2\varrho}\right) |h|^{\varrho}, \qquad x \in \mbb{R}^d,\, |h| \leq 1.
	\end{equation*}
	Define iteratively $\Delta_h u(x) := u(x+h)-u(x)$ and $\Delta_h^k u(x):= \Delta_h (\Delta_h^{k-1} u)(x)$, $k \geq 2$, then the previous inequality shows \begin{equation*}
		|\Delta_h^2 u(x)| \leq 4 C^2 M^2 \left(1+|x|^{\gamma-2\varrho}\right) |h|^{2\varrho}, \qquad x \in \mbb{R}^d, \,|h| \leq 1.
	\end{equation*}
	Iterating the procedure, we find that \begin{equation*}
		|\Delta_h^k u(x)| \leq (2CM)^k \left(1+|x|^{\gamma-k \varrho}\right) |h|^{k \varrho}, \qquad x \in \mbb{R}^d,\, |h| \leq 1,
	\end{equation*}
	for the largest integer $k \geq 1$ such that $\gamma-k \varrho \geq 0$; the latter condition ensures that the constant $\gamma$ in Lemma~\ref{p-5} is non-negative. Applying once more Lemma~\ref{p-5}, we get 
	\begin{equation*}
		|\Delta_h^k u(rx'+rh')-\Delta_h^k u(rx')| \leq (2CM)^{k+1} r^{\gamma-k\varrho } |h'|^{\varrho} |h|^{k\varrho}, \qquad |x'|,|h'| \leq 1, \, r \geq 1.
	\end{equation*}
	If $x,h \in \mbb{R}^d$ are such that $|x| \geq 1$ and $|h| \leq 1$, then we obtain from this inequality for $r=|x|$, $x'=x/r$ and $h'=h/r$ that \begin{equation*}
		|\Delta_h^k u(x+h)-\Delta_h^k u(x)|
		\leq (2CM)^{k+1} |x|^{\gamma-(k+1)\varrho} |h|^{(k+1)\varrho}.
	\end{equation*}
	As $\gamma-(k+1)\varrho<0$, this gives \begin{equation*}
		\sup_{|x|>r} |\Delta_h^{k+1} u(x)|
		 \leq (2CM)^{k+1} r^{\gamma-(k+1)\varrho} |h|^{(k+1) \varrho}
		\xrightarrow[]{r \to \infty} 0.
	\end{equation*}
	Consequently, $x \mapsto w(x):=\Delta_h^{k+1} u(x)$ is for each fixed $|h| \leq 1$ a continuous function which vanishes at infinity and which satisfies $P_1 w=w$. The Liouville property, cf.\ Proposition~\ref{p-9}, yields $w=0$, i.e.\ $\Delta_h^{k+1} u(x)=0$ for all $x \in \mbb{R}^d$ and $|h| \leq 1$.
	We claim that this implies that $u$ is a polynomial. Take $\varphi \in C_c^{\infty}(\mbb{R}^d)$ with $\varphi \geq 0$ and $\int_{\mbb{R}^d} \varphi(x) \, dx=1$, and set $\varphi_n(x) := n^d \varphi(nx)$. The convolution $u_n := u \ast \varphi_n$ satisfies $\Delta_h^{k+1} u_n(x)=0$ for all $x \in \mbb{R}^d$ and $|h| \leq 1$. Since $u_n$ is smooth, we have \begin{equation*}
		\partial_{x_j}^{k+1} u_n(x) = \lim_{r \downarrow 0} \frac{\Delta_{r e_j}^{k+1} u_n(x)}{r^{k+1}} = 0
	\end{equation*}
	for all $x \in \mbb{R}^d$, $j\in\{1,\ldots,d\}$ and $n \in \mbb{N}$; here $e_j$ denotes the $j$-th vector in $\mbb{R}^d$. Hence, $\partial^{\alpha} u_n=0$ for all $|\alpha| \geq N:=(k+1)d$, and so $u_n$ is a polynomial of degree at most $N$ for each $n \in \mbb{N}$. Since $u_n$ converges pointwise to $u$, it follows that $u$ is a polynomial of degree at most $N$.  
	Recalling that $u$ satisfies by assumption the growth condition $|u(x)| \leq M(1+|x|^{\gamma})$ for all $x \in \mbb{R}^d$, we conclude that $u$ is a polynomial of order at most $\floor{\gamma}$.
\end{proof}
	
\begin{ack}	
	I am grateful to Prof.\ Niels Jacob for his comments which helped to improve the presentation of this paper.
\end{ack}

\end{document}